\documentclass[10pt,reqno]{amsart}

\usepackage{amsmath, amssymb}
\usepackage{enumerate}

\usepackage[bitstream-charter]{mathdesign}
\usepackage[T1]{fontenc}

\usepackage{amsthm}
\newtheorem{theo}{Theorem}[section]
\newtheorem{defi}[theo]{Definition}
\newtheorem{prop}[theo]{Proposition}

\newtheorem{lemm}[theo]{Lemma}
\newtheorem{rem}[theo]{Remark}

\newcommand{\al}{\alpha}
\newcommand{\be}{\beta}

\newcommand{\Ga}{\Gamma}

\newcommand{\om}{\omega}
\newcommand{\Om}{\Omega}

\newcommand{\si}{\sigma}

\newcommand{\ep}{\epsilon }
\newcommand{\te}{\theta}
\newcommand{\De}{\Delta}
\newcommand{\de}{\delta}

\newcommand{\R}{{\mathbb R}^n}
\newcommand{\ri}{\rightarrow}
\newcommand{\Rn}{{\mathbb R}^{n-1}}

\newcommand{\na}{\nabla}

\newcommand{\Rb}{{\mathbb R}^n }

\newcommand{\E}{{\mathbb E}}
\newcommand{\N}{{\mathbb N}}
\newcommand{\bP}{{\mathbb P}}
\newcommand{\Z}{{\mathbb Z}}

\newcommand{\bR}{{\mathbb R}}

\newcommand{\cC}{{\mathcal C}}
\newcommand{\cF}{{\mathcal F}}
\newcommand{\cL}{{\mathcal L}}
\newcommand{\cP}{{\mathcal P}}
\newcommand{\cS}{{\mathcal S}}
\newcommand{\cX}{{\mathcal X}}

\newcommand{\abs}[1]{{\left\vert #1 \right\vert}}
\newcommand{\inp}[1]{{\left\langle #1 \right\rangle}}
\newcommand{\norm}[1]{{\left\Vert #1 \right\Vert}}
\newcommand{\set}[1]{{\left\lbrace #1 \right\rbrace}}

\newcommand{\Rnp}{{\mathbb R}^{n}_{+}}

\newcommand{\qand}{\quad \text{ and } \quad}
\newcommand{\qfor}{\quad \text{ for }}

\newcommand{\intzi}{\int_{0}^{\infty}}
\newcommand{\intRn}{\int_{{\mathbb R}^n}}
\newcommand{\sumjz}{\sum_{j \in \Z}}
\newcommand{\supjz}{\sup_{j \in \Z}}

\newcommand{\p}{\partial}

\DeclareMathOperator*{\divg}{div \,}

\begin{document}
\baselineskip=18pt

\title[Stochastic Navier--Stokes equations]{On initial-boundary value problem of the stochastic Navier--Stokes equations in the half space}

\author{Tongkeun Chang \and Minsuk Yang}

\address{Yonsei University, Department of Mathematics, Mathematics, Yonsei University, 50 Yonseiro, Seodaemungu, Seoul, Republic of Korea}

\email{chang7357@yonsei.ac.kr, m.yang@yonsei.ac.kr}

\begin{abstract}
We study the initial-boundary value problem of the stochastic Navier--Stokes equations in the half space.
We prove the existence of weak solutions in the standard Besov space valued random processes when the initial data belong to the critical Besov space.
\end{abstract}

\maketitle

\section{Introduction}
\label{S1}

In this paper we study the following stochastic Navier--Stokes equations
\begin{equation}
\label{E11}
du(t,x) = \big(\De u(t,x) - \na p(t,x) + \divg (u\otimes u)(t,x) \big) dt + g(t,x) dB_t
\end{equation}
for $(t,x) \in (0,\infty) \times \Rnp$, $n \ge 2$, with the boundary condition $u(t,x) = 0$ for $(t,x) \in (0,\infty) \times \Rn$, the initial condition $u(0,x) = u_0$ for $x\in \Rnp$, and the random noise $g(t,x) d B_t$.
Here $\set{B_t(\omega) : t\ge 0,\omega\in\Omega}$ denotes an $n$-dimensional Wiener process defined on a probability space $\Omega$.
The Navier--Stokes system is considered as a reliable model both in science and engineering.
In spite of huge effort done by countless mathematicians, many mathematical issues related to the Navier--Stokes equations including uniqueness and regularity remain as open questions.
Moreover, one of the important issues is to understand turbulence in fluid motions.
Although turbulent phenomena can be frequently observed in a lot of real world situations, it is extremely hard to  define or characterize them mathematically.
Because chaotic behavior in fluid motions may be interpreted as a presence of randomness, many interests on the stochastic Navier--Stokes equations have been increased.

The study of the stochastic Navier--Stokes equations started around early 70s', for example, by Bensoussan and Temam \cite{MR0348841} and Foias \cite{MR352733}.
Although many mathematicians made important contributions, we list here only a few of them.
Capi\'{n}ski and Gatarek \cite{MR1305062} obtained an existence theorem for the stochastic Navier--Stokes equations in a Hilbert space setting.
Mikulevicius and Rozovskii \cite{MR2118862} proved the existence of a global weak (martingale) solution of the stochastic Navier--Stokes equation.
Kim \cite{MR2815038} established the existence of local strong solutions to the stochastic Navier--Stokes equations in $\bR^3$ when the initial data are sufficiently small and slightly more regular.
Taniguchi \cite{MR2837686} studied the energy solutions to the stochastic Navier--Stokes equations in two dimensional unbounded domains.
For a recent overview, we refer the reader to the lecture notes or monographs by Flandoli \cite{MR2459085}, Kuksin and Shirikyan \cite{MR3443633}, and Breit, Feireisl, Hofmanov\'{a} \cite{MR3791804} and the references cited therein.

For the deterministic Navier--Stokes equations, Fujita and Kato \cite{MR166499} proved existence of mild solutions.
Many papers continued to study such solutions with wider class of initial data, which should be in some special classes due to the scaling structure of the Navier--Stokes equations.
Adapting harmonic analysis tools, for example, Littlewood--Paley theory, Fujita--Kato's result extended to the refined function space setting like Besov spaces.
For this, we refer the reader the monograph by Cannone \cite{MR1688096}.
After that, Koch and Tataru \cite{MR1808843} extends the class of initial data to $BMO^{-1}$.
This line of development is well illustrated in the monograph by Bahouri, Chemin, and Danchin \cite{MR2768550}.

Recently, Du and Zhang \cite{MR4002154} proved existence of solutions to the stochastic Navier--Stokes equations in the whole space.
Due to a technical reason they found solutions in Chemin--Lerner type spaces with the norm defined by a weighted sum of space-time integration of Littlewood--Paley projections.
Our aim of this paper is to establish existence of solutions to the stochastic Navier--Stokes equations in the half space.
We found our solutions in the standard Besov spaces.

To state precise concept of solutions we need to recall some standard definitions.
Let $(\Om, {\mathcal G},\{{\mathcal G}_t\}, \bP)$ be a probability space, where $\{{\mathcal G}_t : t\geq 0\}$ is a filtration of $\si$-fields ${\mathcal G}_t \subset {\mathcal G}$ with $\mathcal{G}_0$ containing all $\bP$-null subsets of $\Om$.
Assume that $w_{\cdot}$ is a one-dimensional $\{\mathcal G_t \}$-adapted Wiener processes defined on $(\Om, {\mathcal G}, \bP)$.
We denote by $\E X$ the expectation of a random variable $X$.
We suppress the argument $\omega\in\Omega$ of random variable $X(\omega)$ when it brings no confusion.

We note that the solution $u$ and data $(u_0,g)$ in (\ref{E11}) are random variables.
We construct suitable spaces for them using the Besov spaces.
We shall use two different types of spaces.
The first type emphasizes the regularity in $x$ whereas the second type does the joint regularity in $(t,x)$.
We can consider $u$ and $g$ as Banach space-valued stochastic processes.
Hence $(\Omega\times(0,\infty),\mathcal{P}, \bP \bigotimes\ell(0,\infty))$ is a suitable choice for their common domain, where $\cP$ is the predictable $\si$-field generated by $\{{\mathcal G}_t : t \geq 0\}$ (see, for instance, pp. 84--85 of \cite{MR1661766}) and $\ell(0,\infty)$ is the Lebesgue measure on $(0,\infty)$.

For a Banach space $X$ and $\al \in {\mathbb R}$ we define the stochastic Banach spaces $\cL^r_\al( \Om\times (0,T), \cP; X)$ to be the space of $X$-valued processes with the norm
\[
\norm{f}_{\cL ^r_\al( \Om\times (0,T), \cP; X)}
= \Bigg(\E \int_0^T s^{\al r} \norm{f(s,\cdot)}_X^rds \Bigg)^{1/r}.
\]
Here is the definition of solutions.

\begin{defi}[Local solution]
\label{D21}
Let $n \ge 2$ and $2 \leq p, q \leq \infty$.
Assume that $u_0$ is a $\dot B^{-1 +\frac{n}p}_{pq0}(\Rnp)$-valued $\cF_0$-measurable random variable.
We say that $(u, \tau)$ is a local solution for the stochastic Navier-Stokes equations \eqref{E11} if $u \in \cL ^q_\al(\Om \times (0, \infty), \cP ; L^p(\Rnp))$
is a progressively measurable process and for any $0 \leq t < \infty$ and
\[
\tau(\om) = \inf \set{0 \le T \le \infty : \norm{u(\om)}_{L^q_\al(0, T; L^p(\Rnp))} \ge R}
\]
where $R $ is a positive number, and for almost surely, $u \in L^q_\al (0, \tau; L^p(\Rnp))$ satisfies
\begin{align*}
&\int_{\Rnp} u(x, t\wedge \tau) \Phi(x) dx - \inp{u_0, \Phi(\cdot,0)} \\
&= \int^{t\wedge \tau}_0 \int_{\Rnp} \left(u\cdot \Delta \Phi + u\cdot (\Phi_t+(u\otimes u):\nabla \Phi)\right) dxdt
+ \int_0^{t\wedge \tau} \inp{g,\Phi} dB_t
\end{align*}
for every $\Phi \in C^\infty_0(\overline{\R}_+\times [0,\tau(\om))$ with $\mbox{div}_x\,\Phi=0$, $\Phi\big|_{x_n=0}=0$,
where $\inp{\cdot,\cdot}$ denotes the duality pairing.
\end{defi}

The exact definitions and notations will be given in the next section.
Here are our main theorems.

\begin{theo}
\label{T1}
Let $n < p < \infty$, $2 < q \leq \infty$, and $2\al = 1 -\frac{n}p -\frac2q \ge 0$.
Assume $u_0 \in {\mathbb B}_{pq0}^{-1 +\frac{n}p} (\Rnp)$ with $\divg u_0 = 0$ and $g \in  \cL _{\al_2}^{q}(\Om\times (0,\infty), \cP; L^{p_2} (\Rnp))$ with $\divg g = 0$ where $p_2$ and $\al_2$ satisfy
\[
p_2 \leq p, \quad \al \leq \al_2, \quad \bigg(\frac{n}{p_2}-\frac{n}{p}\bigg) + 2(\al_2-\al) = 1.
\]
Given a positive number $\ep$, there exists a positive number $\delta$ such that if
\begin{align*}
\|u_0\|_{ {\mathbb B}_{pq0}^{-1 +\frac{n}p} (\Rnp)} + \| g \|_{\cL _{\al_2}^{q_2}( \Om\times (0,\infty), \cP; L^{p_2} (\Rnp))} < \de,
\end{align*}
then there is a local unique solution $(u,\tau)$ of the initial boundary value problem \eqref{E11} in the sense Definition \ref{D21} such that $\bP (\tau > 0) =1$ and $\bP(\tau =\infty) \ge 1 -\ep$.
\end{theo}

\begin{theo}
\label{T2}
In addition, if $g \in \cL^q(\Om\times (0,\infty), \cP; \dot B_{pq0}^{-2\al-1}(\Rnp))$, then the solution $u$ obtained in Theorem \ref{T1} belongs to $L^q(0, \tau(\om); \dot B^{-2\al}_{pq} (\Rnp))$ almost surely.
\end{theo}

\begin{rem}
\begin{enumerate}
\item
Both theorems contains the limiting exponent $q =\infty$.
\item
The standard Besov spaces are used throughout the paper (cf. \cite{MR4002154}).
\item
The regularity exponent of the initial data is negative.
\item
The exponents in the main theorems are rather complicated and restrictive due to the Hardy--Littlewood type embedding and scaling structure of governing equations.
\end{enumerate}
\end{rem}

The paper is organized as follows.
In Section \ref{S2} we present preliminaries including a few estimates for the Newtonian potential and the Heat kernel and a representation of the Helmholtz projection operator.
In Section \ref{S3} we prove the existence of the solutions to the linearized Navier--Stokes equations in half space.
In Section \ref{S4} and \ref{S5} we prove Theorem \ref{T1} and \ref{T2}, respectively.

\section{Preliminaries}
\label{S2}
\setcounter{equation}{0}

Let $\cS(\R)$ denote the Schwartz space on $\R$ and $\cS'(\R)$ the space of tempered distributions, which is the dual space of $\cS(\R)$.
The Fourier transform of $f \in \cS(\R)$ is initially defined by
\[
\widehat{f}(\xi) = \int_{\R} e^{-2\pi i x \cdot \xi} f(x) dx,
\]
and extended to the dual space $\cS'(\R)$.

\begin{defi}[Besov spaces on $\R$]
Let $\Psi \in \cS(\R)$ be a radial Schwartz function whose Fourier transform is
nonnegative, supported in the ball $|\xi|\le 2$, equal to $1$ in the ball
$|\xi|\le1$.
Let
\[
\widehat\Phi(\xi)=\widehat\Psi(\xi)-\widehat\Psi(2\xi).
\]
We denote the Littlewood-Paley projection operator by 
\[
\Delta_jf(x)
=\intRn e^{2\pi ix\cdot\xi}\widehat\Phi(2^{-j}\xi)\widehat f(\xi)\;d\xi \qfor j \in \Z.
\]
Notice that 
\[
\sum_{j \in \Z} \widehat\Phi(2^{-j}\xi) = 1, \qquad \text{ for } \xi \neq 0.
\]
We define for $s \in \bR$ and $1 \leq p, q \le \infty$, the homogeneous Besov space $\dot B^s_{pq}(\R)$, which is the set of $f \in \cS'(\R)$ with the norm
\[
\norm{f}_{\dot B^s_{pq} (\R)}
= 
\begin{cases}
\Bigg( \sum_{j \in \Z} (2^{sj} \norm{\De_j f}_{L^p(\R)})^q \Bigg)^{1/q} < \infty & \qfor q<\infty, \\
\sup_{j \in \Z} 2^{sj} \norm{\De_j f}_{L^p(\R)} < \infty & \qfor q=\infty.
\end{cases}
\]
\end{defi}

%

\begin{defi}[Besov spaces on $\Rnp$]
For $1 \leq p, q \le \infty$ and $s \in \bR$, the space
\[
\dot B^s_{pq}(\Rnp) = \{ F|_{\Rnp} : F \in \dot B^s_{pq}(\Rb)\}
\]
is defined as the set of restrictions of $F \in \dot B^s_{pq}(\R)$ with the norm
\[
\norm{f}_{\dot B^s_{pq}(\Rnp)} = \inf \set{\norm{F}_{\dot B^s_{pq}(\Rb)} : F \in \dot B^s_{pq}(\Rb), ~ F|_{\Rnp} =f}.
\]
For $1 \leq p, q < \infty$ and $s \in \bR$, the space ${\dot B^s_{pq0}}(\Rnp)$ is the closure of $C^\infty_c(\Rnp)$ in $\dot B^s_{pq}(\Rnp)$.

We denote by ${\mathbb B}_{pq0}^{-1 +\frac{n}p} (\Rnp)$ the stochastic Besov space with the norm
\[
\norm{f}_{{\mathbb B}_{pq0}^{-1 +\frac{n}p} (\Rnp)} = \E \norm{f}_{B_{pq0}^{-1 +\frac{n}p} (\Rnp)}
\]
\end{defi}
		
\begin{rem}
Let $0 \le s < \infty$, $1 \leq p, q \le \infty$, and let $p', q'$ denote the H\"older conjugates of $p$, $q$, respectively.
Then
\[
(\dot B^s_{p'q'}(\Rnp))' = \dot B^{-s}_{pq0} (\Rnp) \qand (\dot B^s_{p'q'0}(\Rnp))' = \dot B^{-s}_{pq}(\Rnp).
\]
\end{rem}

\begin{defi}
Let $X$ be a Banach space and $I$ an interval in $\bR$.
For $1 \le r \le \infty$ and $0 \le \al < \infty$, we denote by $L^r_\al(I;X)$ the weighted Bochner space with the norm
\[
\norm{f}_{L^r_\al(I;X)} = \Bigg(\int_I s^{\al r} \norm{f(s)}_{X}^r ds\Bigg)^{\frac{1}{r}}.
\]
In particular, if $\al =0$, then it is the same as the usual Bochner space, i.e., $L^r_0(I;X) = L^r(I;X)$.
\end{defi}

\begin{defi}[Newtonian potential]
The fundamental solution of the Laplace equation in $\R$ is defined by
\[
N(x) = \left\{\begin{array}{ll}
\vspace{2mm}
\frac{1}{\sigma_n (2-n)|x|^{n-2}} & \mbox{if } n\geq 3,\\
\frac{1}{2\pi}\ln |x| & \mbox{if } n=2.\end{array}\right.
\]
Here $\sigma_n$ denotes the surface area of the unit sphere in $\R$.
\end{defi}

\begin{lemm}[Lemma 3.3 in \cite{MR3959494}]
\label{L21}
If  $1 < p < \infty$ and  $0 \le s < \infty$, then
\[
\norm{\nabla_x^2 \int_{\Rnp} N(\cdot-y) f(y) dy}_{\dot H^{s}_{p}(\Rnp)}
\lesssim \norm{f}_{\dot H^{s}_{p}(\Rnp)},
\]
where $\dot H^{s}_{p}(\Rnp)$ is the space of all restriction of $f \in \dot H^{s}_{p}(\R)$ on $\Rnp$, which is the standard homogenous factional Sobolev space defined by the Fourier inversion.
(In particular, $\dot H^0_{p}(\Rnp) = L^p(\Rnp)$.)
\end{lemm}

\begin{lemm}
\label{L22}
If $1 < p < \infty$, $1 \le q \le \infty$, and  $0 < s <  1/p'$, then
\begin{align*}
\norm{\nabla_x^2 \int_{\Rnp} N(\cdot-y) f(y) dy}_{\dot B^{-s}_{pq} (\Rnp)}
&\lesssim \norm{f}_{ \dot B^{-s}_{pq}(\Rnp)}, \\
\norm{\nabla_x^2 \int_{\Rnp} N(\cdot-y^*) f(y) dy}_{\dot B^{-s}_{pq} (\Rnp)}
&\lesssim \norm{f}_{ \dot B^{-s}_{pq}(\Rnp)}.
\end{align*}
where $p'$ is the H\"older conjugate of $p$, i.e., $1/p + 1/p' = 1$.
\end{lemm}

\begin{proof}
Applying the real interpolation (see for example Section 2 in \cite{CJ}) to Lemma \ref{L21} we get
\[
\norm{\nabla_x^2 \int_{\Rnp} N(\cdot-y) f(y) dy}_{ \dot B^{s}_{p'q'} (\Rnp)}
\lesssim \norm{f}_{ \dot B^{s}_{p'q'}(\Rnp)}.
\]
Since the operator
\[
Tf(x) := \nabla_x^2 \int_{\Rnp} N(x-y) f(y) dy
\]
is symmetric, we use the duality argument to obtain
\[
\norm{\nabla_x^2 \int_{\Rnp} N(\cdot-y) f(y) dy}_{ \dot B^{-s}_{pq0} (\Rnp)}
\lesssim \norm{f}_{ \dot B^{-s}_{pq0}(\Rnp)}.
\]
We note that if $-1/p < s < 1/p'$, then
\[
\dot B^{-s}_{pq0} (\Rnp) = \dot B^{-s}_{pq}(\Rnp).
\]
This proves the first estimate and the proof of the second estimate is almost the same.
\end{proof}

\begin{defi}[Heat kernel]
The fundamental solution of the heat equation in $\R$ is defined by
\[
\Gamma_t(x) = \Gamma(x,t) =
\begin{cases}
(2\pi t)^{-\frac{n}{2}} e^{-\frac{|x|^2}{4t}} & \text{ if } t>0 \\
0 & \text{ if } t \le 0.
\end{cases}
\]
\end{defi}

We define
\begin{align*}
\Gamma_t * f(x) &= \intRn\Gamma(x-y,t)f(y) dy, \\
\Gamma_t^* * f(x) &= \intRn\Gamma(x-y^*,t)f(y) dy.
\end{align*}
The following lemma is a consequence of a pointwise estimate of the localized heat kernel,
which is obtained by the standard integration by parts argument.

From now on we simply write $\norm{f}_p$ instead of $\norm{f}_{L^p(\R)}$.

\begin{lemm}
\label{L23}
There exists a positive constant $c$ such that for all $1 \le p \le \infty$, $t > 0$, and $j \in \Z$
\[
\norm{\De_j (\Ga_t * f)}_p
\lesssim \exp(-ct2^{2j}) \norm{\De_j f}_p.
\]
\end{lemm}

\begin{proof}
Fix $t > 0$.
For $\lambda>0$ we define
\[
K_\lambda(x,t)
= \intRn e^{2\pi ix\cdot\xi} \widehat{\Psi}(\xi/\lambda)
e^{-t|\xi|^2} d\xi.
\]
By a change of variables we have
\[
K_\lambda(x,t) = \lambda^n \intRn e^{2\pi i\lambda x\cdot\xi}
\widehat{\Psi}(\xi) e^{-t\lambda^2 |\xi|^2} d\xi.
\]
Using the identity
\[
(I-\Delta_\xi)e^{2\pi i\lambda x\cdot\xi}
=(1+4\pi^2 |\lambda x|^2)e^{2\pi i\lambda x\cdot\xi}
\]
we can carry out repeated integration by parts to get the pointwise estimate of the localized heat kernel
\[
K_\lambda(x,t)
\lesssim \frac{\lambda^n \exp(-c_1t\lambda^2)}{(1+|\lambda x|^2)^{(n+1)/2}}
\]
for some positive constant $c_1$.
Using a change of variable yields
\begin{equation}
\label{E21}
\int_{{\mathbb R}^n} K_\lambda(x,t) dx \lesssim \exp(-c_1t\lambda^2).
\end{equation}
The implied constant does not depend on $\lambda$.
Since $\De_j = (\De_{j-1} + \De_j + \De_{j+1}) \De_j$ and the convolutions commute, we have
\[
\De_j (\Ga_t * f) = (\De_{j-1} + \De_j + \De_{j+1}) \Ga_t * (\De_j f).
\]
Applying Young's inequality and \eqref{E21} we obtain that
\begin{align*}
\norm{\De_j (\Ga_t * f)}_p
&\le \left(\norm{K_{2^{j-1}}}_1 + \norm{K_{2^{j}}}_1 + \norm{K_{2^{j+1}}}_1\right) \norm{\De_j f}_p \\
&\le \left(\exp(-c_1t2^{2j-2}) + \exp(-c_1t2^{2j}) + \exp(-c_1t2^{2j+2})\right) \norm{\De_j f}_p \\
&\le 3\exp(-ct2^{2j}) \norm{\De_j f}_p
\end{align*}
where $c = c_1/4$.
\end{proof}

\begin{lemm}
\label{L24}
If $1 \le p, q \le \infty$ and $\be \in {\mathbb R}$, then
\[
\norm{\Ga_t * f}_{L^q(0, \infty;  \dot B^{\be}_{pq} (\R))}
+ \norm{\Ga_t^* * f}_{L^q(0, \infty;  \dot B^{\be}_{pq} (\R))}
\lesssim \norm{f}_{\dot B^{\be -\frac2q}_{pq}(\R)}.
\]
\end{lemm}

\begin{proof}
Fix $1 \le p \le \infty$.
If $1 \le q < \infty$, then by Lemma \ref{L23}
\begin{align*}
\norm{\Ga_t * f}_{L^q(0, \infty; \dot B^{\be}_{pq} (\R)) }^q
&=\intzi \sumjz 2^{q\be j} \norm{\De_j (\Ga_t * f)}_p^q dt \\
&\lesssim \intzi \sumjz 2^{q\be j} \exp(-ct2^{2j}q) \norm{\De_j f}_p^q dt \\
&= \sumjz 2^{q\be j} \norm{\De_j f}_p^q \intzi \exp(-ct2^{2j}q) dt \\
&= \frac{1}{cq} \sumjz 2^{(q \be-2) j} \norm{\De_j f}_p^q \\
&= \frac{1}{cq} \norm{f}_{\dot B^{\be -\frac2q}_{pq}(\R)}^q.
\end{align*}
Note that the operator norm does not depend on $q$ since $\lim_{q \to \infty} (cq)^{1/q} = 1$.
If $q = \infty$, then by Lemma \ref{L23}
\begin{align*}
\norm{\Ga_t * f}_{L^\infty(0, \infty;  \dot B^{\be}_{p\infty} (\R)) }
&=\sup_{t>0} \supjz 2^{\be j} \norm{\De_j (\Ga_t * f)}_p \\
&\lesssim \supjz \sup_{t>0} 2^{\be j} \exp(-ct2^{2j}) \norm{\De_j f}_p \\
&= \supjz 2^{\be j} \norm{\De_j f}_p \\
&= \norm{f}_{\dot B^{\be}_{p\infty}(\R)}.
\end{align*}
The estimates for $\Ga_t^* * f$ can be done by exactly the same way.
\end{proof}

\begin{lemm}
\label{L25}
Let $1 \leq p, q \le \infty$ and $0 < \al < \infty$.
Then
\[
\norm{\Ga_t * f}_{L^q_{ \al} (0, \infty; L^p (\R) ) }
+ \norm{\Ga_t^* * f}_{L^q_{ \al }(0, \infty; L^p(\R) ) }
\lesssim \norm{f}_{\dot B^{ -2\al-2/q}_{pq}(\R)}.
\]
\end{lemm}

\begin{proof}
First we consider the case $q = \infty$.
Using the identity $f = \sumjz \De_j f$ and the triangle inequality and applying Lemma \ref{L23}, we get
\[
t^{\al} \norm{\Ga_t * f}_p
\le t^{\al} \sumjz \norm{\De_j(\Ga_t * f)}_p
\lesssim t^{\al} \sumjz \exp(-ct2^{2j}) \norm{\De_j f}_p.
\]
Since $2^{-2\al j} \norm{\De_j f}_p \le \norm{f}_{\dot B^{-2\al}_{p \infty}(\R)}$ for all $j$, we have
\[
t^{\al} \sumjz \exp(-ct2^{2j}) \norm{\De_j f}_p
\lesssim \norm{f}_{\dot B^{-2\al}_{p \infty}(\R)} \sumjz (t2^{2j})^\al \exp(-ct2^{2j})
\lesssim \norm{f}_{\dot B^{-2\al}_{p \infty}(\R)} .
\]
Thus, we have for all $t$,
\[
t^{\al} \norm{\Ga_t * f}_p
\lesssim \norm{f}_{\dot B^{-2\al}_{p \infty} (\R)}.
\]
This proves the lemma for the case $q=\infty$.

Now we consider the case $q=1$.
By the same way
\begin{align*}
\intzi t^{\al} \norm{\Ga_t * f}_p dt
&\le \intzi t^{\al} \sumjz \norm{\De_j(\Ga_t * f)}_p dt \\
&\lesssim \sumjz \norm{\De_j f}_p \intzi t^{\al} \exp(-ct2^{2j}) dt \\
&\lesssim \sumjz 2^{-2\al j - 2j} \norm{\De_j f}_p \intzi (t2^{2j})^{\al} \exp(-ct2^{2j}) 2^{2j} dt \\
&\lesssim \norm{f}_{\dot B^{ -2\al-2}_{p1}(\R)}
\end{align*}
since the integral $\intzi (t2^{2j})^{\al} \exp(-ct2^{2j}) 2^{2j} dt$ does not depend on $j$ by a change of variable.
This proves the lemma for the case $q=1$.
Using complex interpolation (5.6.3. Theorem in \cite{MR0482275})
\[
[L_\al^\infty(0,\infty; L^p(\R)), L_\al^1(0,\infty; L^p(\R))]_\theta = L_\al^q(0,\infty; L^p(\R))
\]
and (6.4.5. Theorem in \cite{MR0482275})
\[
[\dot B_{p\infty}^{-2\al}(\R),\dot B_{p1}^{-2\al-2}(\R)]_\theta = \dot B_{pq}^{-2\al-2/q}(\R)
\]
with $1/q = 1-\theta$, we get the result.
The estimates for $\Ga_t^* * f$ can be done by exactly the same way.
\end{proof}

Here we recall the following representation of the Helmholtz projection operator.
See Section 3 of  \cite{MR2120798}  for more details.

\begin{defi}[Helmholtz projection $\bP$ in $\Rnp$]
Let $\ \cF=(F_{kl})_{k,l=1}^n$ with
\[
F_{kl}=F_{lk} \qand F_{mk}|_{x_n=0}=0.
\]
If $f=\mbox{div}\cF$, then the Helmholtz projection operator $\bP f$ in $\Rnp$ is defined by
\[
\bP f=\mbox{div} \, \cF',
\]
where ${\mathcal F'}=(F_{km}')_{k,m=1}^n$ is given by
\begin{align*}
F'_{nm} &=F_{nm}-\delta_{nm}F_{nn}, \ m=1,\cdots, n,\\
F'_{\beta \gamma} &=F_{\beta \gamma}-\delta_{\beta \gamma}F_{nn} +
\sum_{q=1}^n D_{x_\gamma} \int_{\Rnp}D_{y_q}N^+(x,y)F_{\beta q}(y) dy \\
&\quad+ D_{x_\gamma} \int_{\Rnp} \Big( D_{y_n}N^+(x,y)F_{n\beta}(y)-D_{y_\beta}N^+(x,y)F_{nn}(y)\Big) dy, \quad \beta,\gamma\neq n,\\
F'_{\beta n} &=-\sum_{\gamma=1}^{n-1}D_{x_\gamma}\int_{\Rnp}D_{x_n}N^+(x,y)F_{\beta \gamma}(y) dy+D_{x_\beta}\int_{R_+}D_{x_n}N^+(x,y) F_{nn}(y)dy\\
&\quad-2F_{\beta n}(x)-2\sum_{\gamma=1}^{n-1}D_{x_\gamma}\int_{\Rnp}D_{x_\gamma}N^{-}(x,y)F_{\beta n}(y) dy \quad \beta\neq n.
\end{align*}
Here we used shorthand notations
\begin{align*}
N^+(x,y) &:= N(x-y)+N(x-y^*), \\
N^{-}(x,y) &:= N(x-y)-N(x-y^*).
\end{align*}
\end{defi}

\begin{rem}
By Lemma \ref{L21} for $1 < p < \infty$
\begin{align}
\label{E22}
\|\cF'\|_{L^p (\Rnp)}& \lesssim \|\cF\|_{L^p (\Rnp)}.
\end{align}
\end{rem}

\begin{lemm}[Hardy-Littilewood-Sobolev inequality \cite{MR1328645}]
\label{L26}
Let $0<\lambda<1$ and define for $t > 0$
\[
I_\lambda f(t) = \int_0^t(t-s)^{-\lambda}f(s) ds.
\]
Then the operator $I_\lambda : L_\al^p(0, \infty) \to L^{q}_{\be}(0, \infty)$ is bounded if $1 < p \leq q < \infty$ and $\be \le \al$ satisfy
\[
0 < \al p + 1 < p, \qquad
0 < \be q + 1 < q, \qquad
1 + \frac1q + \be = \frac1p + \lambda + \al.
\]
\end{lemm}

%

\section{ Linearized Navier-Stokes equations}
\label{S3}
\setcounter{equation}{0}

In this section, we consider the initial-boundary value problem of the stochastic linearized Navier--Stokes equations in half space
\begin{equation}
\label{E31}
d u(t,x) = \big(\De u(t,x) -\na p(t,x) + {\rm div} \, \cF(t,x)  \big) dt + g(t,x) d B_t
\end{equation}
in $\Om \times (0,\infty) \times \Rnp$ with the boundary condition $u|_{x_n =0} = 0$ and the initial condition $u|_{t =0} = u_0$.
The goal of this section is to prove the following proposition.

\begin{prop}
\label{P31}
Let $2 \le p < \infty$, $2 < q < \infty$, and $0 \leq 2\al < 1-\frac{1}{p}-\frac{2}{q}$.
Assume $u_0 \in \dot {\mathbb B}^{-2\al-\frac{2}{q}}_{pq0}(\Rnp)$ with $\divg u_0 = 0$, $\cF \in \cL ^{q_1}_{\al_1}(\Om \times (0,\infty),\cP; L^{p_1}(\Rnp))$ for some $p_1$, $q_1$, and $\al_1$ satisfying
\begin{equation}
\label{E32}
q_1 \le q, \quad \al \le \al_1 < 1-\frac{1}{q_1}, \quad
\bigg(\frac{n}{p_1}-\frac{n}{p}\bigg) + \bigg(\frac{2}{q_1}-\frac{2}{q}\bigg) + 2(\al_1-\al) = 1,
\end{equation}
and $g \in  \cL _{\al_2}^{q}( \Om\times (0,\infty), \cP; L^{p_2} (\Rnp))$ for some $p_2$ and $\al_2$ satisfying
\[
p_2 \leq p, \quad \al \leq \al_2, \quad \bigg(\frac{n}{p_2}-\frac{n}{p}\bigg) + 2(\al_2-\al) = 1.
\]
Then there is at least one weak solution $u\in \cL ^q_\al(\Om \times (0,\infty),\cP;L^p(\Rnp))$ to the Stokes equations \eqref{P31} satisfying
\begin{equation}
\label{E33}
\norm{u}_{\cL ^q_\al(\Om \times (0,\infty),\cP;L^p(\Rnp))}
\le C_1 \Big(\|u_0\|_{\dot {\mathbb B}^{-2\al-\frac{2}{q}}_{pq0}(\Rnp)}
+ \norm{g}_{\cL _{\al_2}^{q}( \Om\times (0,\infty), \cP; L^{p_2} (\Rnp))}
+ \norm{\cF}_{\cL ^{q_1}_{\al_1}(\Om \times (0,\infty),\cP; L^{p_1}(\Rnp))}\Big).
\end{equation}
Moreover, the solution $u$ is unique in $\cL ^q_\al(\Om \times (0,\infty),\cP;L^p(\Rnp))$ when $0 \le \al < 1-\frac1q$.
\end{prop}

Before proving the proposition we recall the following solvability lemma for the initial-boundary value problem of the deterministic linearized Navier--Stokes equations in half space.
For the proof we refer the reader to Theorem 1.2 in \cite{MR3474350} for the case $q < \infty$ and to Theorem 1.3 in \cite{CJ} for the case $q =\infty$.

\begin{lemm}
\label{L32}
Let $1 < p < \infty$, $2 < q \leq \infty$, and $0 \leq 2\al < 1-\frac{1}{p}-\frac{2}{q}$.
Assume $u_0 \in \dot B^{-2\al-\frac{2}{q}}_{pq0}(\Rnp)$ with $\divg u_0 = 0$, $\cF \in L^{q_1}_{\al_1}(0,\infty; L^{p_1}(\Rnp))$ for some $p_1$, $q_1$, and $\al_1$ satisfying \eqref{E32}.
Then there is at least one weak solution $u_1\in L^q_\al(0,\infty;L^p(\Rnp))$ to the Stokes equations
\begin{equation}
\label{E34}
\p_t u_1 = \De u_1 - \na p_1 -\mbox{div} \cF \qand \divg u_1 = 0
\end{equation}
in $(0,\infty) \times \Rnp$ with the boundary condition $u_1|_{x_n =0} = 0$ and the initial condition $u_1|_{t =0} = u_0$.
Moreover, the solution satisfies the estimate
\begin{equation}
\label{E35}
\norm{u_1}_{L^q_\al(0,\infty;L^p(\Rnp))}
\lesssim \|u_0\|_{B^{-2\al-\frac{2}{q}}_{pq0}(\Rnp)}+\|\cF\|_{L^{q_1}_{\al_1}(0,\infty; L^{p_1}(\Rnp))}.
\end{equation}
Moreover, the solution $u_1$ is unique in $L^q_\al(0,\infty;L^p(\Rnp))$ when $0 \le \al < 1-\frac1q$.
\end{lemm}

\begin{proof}[Proof of Proposition \ref{P31}]
For each fixed $\omega \in \Omega$, we solve the Stokes equation \eqref{E34} with the initial data $u_0(\omega)$ to obtain the unique solution $u_1$ satisfying the estimate \eqref{E35}.
If we solve the following stochastic Stokes problem
\begin{equation}
\label{E36}
d u_2(t,x) = \big(\De u_2(t,x) -\na p_2(t,x)  \big) dt + g(t,x) d B_t
\end{equation}
in $\Omega \times (0,\infty) \times \Rnp$ with the boundary condition $u_2|_{x_n =0} = 0$ and the initial condition $u_2|_{t =0} = 0$, then we can take $u = u_1 + u_2$.
We note that the solution $u_2$ can be represented by
\[
u_2(x,t) = \int_0^t T_Kg(x,t,s) dB_s,
\]
where
\[
T_Kg(x,t,s) = \int_{\Rnp} K(x,y, t-s) g(y,s) dy
\]
and the kernel $K = (K_{ij})$ is given by
\begin{equation}
\label{E37}
\begin{split}
K_{ij}(x,y,t)
&= \delta_{ij}\Big(\Gamma(x-y,t)-\Gamma(x-y^*,t)\Big) \\
&\quad + 4(1-\delta_{jn}) D_{x_j} \int^{x_n}_0 \int_{\bR^{n-1}}D_{x_i}N(x-z)\Gamma(z-y^*,t)dz
\end{split}
\end{equation}
for $i,j=1,\cdots, n$.
Here $y^*=(y',-y_n)$.
(See \cite{MR0460931} for details.)
We have
\begin{align*}
\E \intzi  t^{\al q} \norm{u_2}_{L^p (\Rnp)}^q  dt
&= \E \intzi  t^{\al q} \left(\int_{\Rnp} \abs{\int_0^t T_Kg(x,t,s) dB_s}^p dx \right)^{\frac{q}p} dt \\
&\lesssim \E \intzi  t^{\al q} \left(\int_{\Rnp} \abs{\int_0^t \abs{T_Kg(x,t,s)}^2 ds}^\frac{p}2 dx \right)^{\frac{q}p} dt \\
&\leq \E \intzi  t^{\al q} \left(\int_0^t \abs{\int_{\Rnp} \abs{T_Kg(x,t,s)}^p dx}^\frac2p ds\right)^\frac{q}2 dt
\end{align*}
by the Burkholder--Davis--Gundy inequality (Section 2.7 in \cite{MR1661766}) and Minkowski's integral inequality.
We now recall the following estimate
\begin{align*}
\int_{\Rnp} \abs{T_Kg(x,t,s)}^p dx
&\lesssim \int_{\Rnp} \abs{\int_{\Rnp} \Ga_{t-s} (x -y) f(y,s) dy}^p dx \\
&\quad+ \int_{\Rnp} \abs{\int_{\Rnp} \Ga_{t-s} (x -y^*) f(y,s) dy}^p dx
\end{align*}
in Section 3.1 in \cite{MR3780493}.
Then we obtain that, by Young's inequality and Lemma \ref{L26},
\begin{align*}
\E \intzi  t^{\al q} \norm{u_2}_{L^p(\Rnp)}^q dt
&\leq \E \intzi  t^{\al q} \left(\int_0^t \abs{\int_{\Rnp} \abs{\int_{\Rnp} \Ga_{t-s}(x-y) g(y,s) dy}^p dx}^\frac2p ds\right)^\frac{q}2 dt \\
&\quad + \E \intzi  t^{\al q} \left(\int_0^t \abs{\int_{\Rnp} \abs{\int_{\Rnp} \Ga_{t-s}(x-y^*) g(y,s) dy}^p dx}^\frac2p ds\right)^\frac{q}2 dt \\
&\leq \E \intzi  t^{\al q} \left(\int_0^t (t -s)^{-\frac{n}{p_2} + \frac{n}p} \norm{g (s)}_{L^{p_2}(\Rnp)}^2  ds\right)^\frac{q}2 dt \\
&\le \E \intzi  t^{\al_2 q} \norm{g}_{L^{p_2}(\Rnp)}^q dt,
\end{align*}
where $\bigg(\frac{n}{p_2}-\frac{n}{p}\bigg) + 2(\al_2-\al) = 1$.
Combining this together with Lemma \ref{L32}, we obtain that
\[
\norm{u}_{\cL ^q_\al(\Om \times (0,\infty),\cP;L^p(\Rnp))}
\lesssim \norm{g}_{\cL _{\al_2}^{q}( \Om\times (0,\infty), \cP; L^{p_2} (\Rnp))}.
\]
Summing this and \eqref{E35} we get \eqref{E33}.
\end{proof}

\section{Proof of Theorem \ref{T1}}
\label{S4}
\setcounter{equation}{0}

We divide the proof into a few steps.

\begin{enumerate}[\bf{Step} 1)]
\item
Define the Banach space $\cX$ and its closed subset $\cC_R$ as
\[
\cX := \cL^q_\al (\Om \times (0,\infty), \cP; L^p(\Rnp)) \qand \cC_R := \set{u \in \cX : \norm{u}_{\cX} \le R},
\]
where the number $0 <R \le 1$ will be determined in later.
Fix $0 \leq \al < \frac1{q'}$ and an initial data $u_0 \in \dot {\mathbb B}^{-2\al-\frac{2}{q}}_{pq0}(\Rnp)$ with $\divg u_0 = 0$.
Then for each $v \in \cC_R$ there exists a unique solution $w(\om)$, for every $\om \in \Om$, of the following problem
\begin{equation}
\label{E40}
d w = \big( \De w  + \na P_v - {\rm div } \, ( \chi_v v \otimes v )\big)dt + g dB_t \qand \divg w = 0
\end{equation}
with the boundary condition $w|_{x_n =0} = 0$ and the initial condition $w|_{t =0} = u_0$ by Proposition \ref{P31}.
Thus, we can define the solution map by
\[
S(v)=w.
\]
We introduced the function $\chi_{v(\om)} : [0,\infty) \to [0,1]$ defined by
\[
\chi_{v(\om)}(t) := \te (\norm{v(\om)}_{L^q_\al (0,t; L^p (\Rnp))}),
\]
where $\te : [0, \infty) \to [0, 1]$ is a continuous decreasing function defined by
\[
\te (s) = \set{0 \vee (2-R^{-1}s)} \wedge 1.
\]
\item
Choose $q_1=\frac{q}{2}$, $p_1=\frac{p}{2}$, $\al_1=2\al$ so that $\bigg(\frac{n}{p_1}-\frac{n}{p}\bigg) + \bigg(\frac{2}{q_1}-\frac{2}{q}\bigg) + 2(\al_1-\al) = 1$.
Hence $-2\al-\frac2q=-1+\frac{n }{p}$.
For every $\om \in \Om$, we get by H\"older's inequality
\[
\norm{\chi_v v \otimes v}_{L^{q_1}_{\al_1}(0,\infty;L^{p_1}(\Rnp))}
\le \norm{\sqrt{\chi_v} v}_{L^q_\al(0,\infty;L^p(\Rnp))}^2
\le 4R^2.
\]
If $v \in \cC_R$, then by using Proposition \ref{P31} and assumption of Theorem \ref{T1}
\begin{align*}
\norm{S(v)}_{\cX}
&= \norm{w}_{\cX} 
\le C_1(\|u_0\|_{\dot {\mathbb B}^{-2\al-\frac{2}{q}}_{pq0}(\Rnp)}
+ \norm{g}_{\cL _{\al_2}^{q}( \Om\times (0,\infty), \cP; L^{p_2} (\Rnp))} + 4R^2) \\
&\le C_1(\delta + 4R^2).
\end{align*}
Now, we can take $\de = R^2$ so that
\begin{equation}
\label{E41}
\norm{S(v)}_{\cX} \le 5 C_1 R^{2}.
\end{equation}
\item
We claim that
\begin{equation}
\label{E42}
|\chi_u - \chi_v| \le R^{-1} \norm{u-v}_{L^q_\al (0,t; L^p (\Rnp))}.
\end{equation}
Note first that $\chi_{v(\om)} (t) = 1$ if $\norm{v(\om)}_{L^q_\al (0,t; L^p (\Rnp))} \le R$ and that $\chi_{v(\om)} (t) = 0$ if $\norm{v(\om)}_{L^q_\al (0,t; L^p (\Rnp))} \ge 2R$.
To prove the claim, we may consider the following six cases.
\begin{enumerate}
\item
If $\norm{u}_{L^q_\al (0,t; L^p (\Rnp))} \le \norm{v}_{L^q_\al (0,t; L^p (\Rnp))} \le R$, then $|\chi_u - \chi_v| = |1-1| = 0$.
\item
If $\norm{u}_{L^q_\al (0,t; L^p (\Rnp))} \le R \le \norm{v}_{L^q_\al (0,t; L^p (\Rnp))} \le 2R$, then
\begin{align*}
|\chi_u - \chi_v|
&= |1-(2-R^{-1}\norm{v}_{L^q_\al (0,t; L^p (\Rnp))})| \\
&= R^{-1}|\norm{v}_{L^q_\al (0,t; L^p (\Rnp))}-R| \\
&\le R^{-1}|\norm{v}_{L^q_\al (0,t; L^p (\Rnp))}-\norm{u}_{L^q_\al (0,t; L^p (\Rnp))}| \\
&\le R^{-1} \norm{u-v}_{L^q_\al (0,t; L^p (\Rnp))}.
\end{align*}
\item
If $\norm{u}_{L^q_\al (0,t; L^p (\Rnp))} \le R \le 2R \le \norm{v}_{L^q_\al (0,t; L^p (\Rnp))}$, then
\begin{align*}
|\chi_u - \chi_v|
&= |1-0| \le \frac{\norm{v}_{L^q_\al (0,t; L^p (\Rnp))}-\norm{u}_{L^q_\al (0,t; L^p (\Rnp))}}{2R-R} \\
&\le R^{-1} \norm{u-v}_{L^q_\al (0,t; L^p (\Rnp))}.
\end{align*}
\item
If $R \le \norm{u}_{L^q_\al (0,t; L^p (\Rnp))} \le \norm{v}_{L^q_\al (0,t; L^p (\Rnp))} \le 2R$, then
\begin{align*}
|\chi_u - \chi_v|
&= |(2-R^{-1}\norm{u}_{L^q_\al (0,t; L^p (\Rnp))})-(2-R^{-1}\norm{v}_{L^q_\al (0,t; L^p (\Rnp))})| \\
&= R^{-1}|\norm{v}_{L^q_\al (0,t; L^p (\Rnp))}-\norm{u}_{L^q_\al (0,t; L^p (\Rnp))}| \\
&\le R^{-1} \norm{u-v}_{L^q_\al (0,t; L^p (\Rnp))}.
\end{align*}
\item
If $R \le \norm{u}_{L^q_\al (0,t; L^p (\Rnp))} \le 2R \le \norm{v}_{L^q_\al (0,t; L^p (\Rnp))}$, then
\begin{align*}
|\chi_u - \chi_v|
&= |(2-R^{-1}\norm{u}_{L^q_\al (0,t; L^p (\Rnp))})-0| \\
&= R^{-1}|2R-\norm{u}_{L^q_\al (0,t; L^p (\Rnp))}| \\
&\le R^{-1}|\norm{v}_{L^q_\al (0,t; L^p (\Rnp))}-\norm{u}_{L^q_\al (0,t; L^p (\Rnp))}| \\
&\le R^{-1} \norm{u-v}_{L^q_\al (0,t; L^p (\Rnp))}.
\end{align*}
\item
If $2R \le \norm{u}_{L^q_\al (0,t; L^p (\Rnp))} \le \norm{v}_{L^q_\al (0,t; L^p (\Rnp))}$, then $|\chi_u - \chi_v| = |0-0| = 0$.
\end{enumerate}
Thus, \eqref{E42} always holds true.
\item
We claim that there exists $C_2>0$ such that
\begin{equation}
\label{E43}
\norm{S(u)-S(v)}_{\cX} \le C_2 R \norm{u-v}_{\cX}.
\end{equation}
Note first that
\[
|\chi_u u \otimes u - \chi_v v \otimes v|
\le |(\chi_u - \chi_v) (u \otimes u)| + |\chi_v (u-v) \otimes u| + |\chi_v v \otimes (u-v)|.
\]
To prove the claim, we may consider the following two cases.
\begin{enumerate}
\item
If $\chi_u(t)>0$ and $\chi_v(t)>0$, then by \eqref{E42}
\begin{align*}
\norm{(\chi_u - \chi_v) (u \otimes u)}_{L^{\frac{q}{2}}_{2\al}(0,t; L^{\frac{p}{2}}(\Rnp))}
&\le R^{-1} \norm{u-v}_{L^q_\al (0,t; L^p (\Rnp))} \norm{u}_{L^q_\al (0,t; L^p (\Rnp))}^2 \\
&\le 4R \norm{u-v}_{L^q_\al (0,t; L^p (\Rnp))}.
\end{align*}
and
\begin{align*}
&\norm{\chi_v (u-v) \otimes u}_{L^{\frac{q}{2}}_{2\al}(0,\infty; L^{\frac{p}{2}}(\Rnp))}
+ \norm{\chi_v v \otimes (u-v)}_{L^{\frac{q}{2}}_{2\al}(0,\infty; L^{\frac{p}{2}}(\Rnp))} \\
&\le \norm{u-v}_{L^q_\al (0,t; L^p (\Rnp))} \Big(\norm{u}_{L^q_\al (0,t; L^p (\Rnp))} + \norm{v}_{L^q_\al (0,t; L^p (\Rnp))}\Big) \\
&\le 4R \norm{u-v}_{L^q_\al (0,t; L^p (\Rnp))}.
\end{align*}
\item
If $\chi_u(t)>0$ and $\chi_v(t)=0$, then by the similar way
\begin{align*}
\norm{(\chi_u - \chi_v) (u \otimes u)}_{L^{\frac{q}{2}}_{2\al}(0,t; L^{\frac{p}{2}}(\Rnp))}
&\le R^{-1} \norm{u-v}_{L^q_\al (0,t; L^p (\Rnp))} \norm{u}_{L^q_\al (0,t; L^p (\Rnp))}^2 \\
&\le 4R \norm{u-v}_{L^q_\al (0,t; L^p (\Rnp))}.
\end{align*}
and $\norm{\chi_v (u-v) \otimes u}_{L^{\frac{q}{2}}_{2\al}(0,\infty; L^{\frac{p}{2}}(\Rnp))}
= \norm{\chi_v v \otimes (u-v)}_{L^{\frac{q}{2}}_{2\al}(0,\infty; L^{\frac{p}{2}}(\Rnp))} = 0$.
\end{enumerate}
Thus, we use Theorem \ref{P31} to obtain that
\begin{align*}
\norm{S(u)-S(v)}_{\cL ^q_\al(0,\infty;L^p(\Rnp))}
&\lesssim \norm{\chi_u u \otimes u - \chi_v v \otimes v}_{\cL ^{\frac{q}{2}}_{2\al}(0,\infty;L^{\frac{p}{2}}(\Rnp))}\\
&\lesssim 4R \norm{u-v}_{L^q_\al (0,\infty; L^p (\Rnp))}.
\end{align*}
This proves the claim \eqref{E43}.
\item
From \eqref{E41} $S(\cC_R) \subset \cC_R$ and from \eqref{E43} the map $S$ is contractive in the complete space $\cC_R$ whenever $R \le \frac{1}{5C_1+C_2}$.
Therefore we obtain a unique solution $u$ to \eqref{E40} in $\cC_R$ by the Banach fixed point theorem.
Now, we define the stopping time
\[
\tau(\om) = \inf \set{0 \le T \le \infty\, : \norm{u}_{L^q_\al (0, T;  L^p (\Rnp))} \ge R}.
\]
If $\tau(\om) < \infty$, then we set $u(t) =0$ for $t \geq \tau(\om)$ so that $(u, \tau)$ is the unique solution in the sense of Definition \ref{D21}.

Notice from definition of $\tau$ that for all $h \in (0,\infty)$,
\[
\{\om \, | \, \tau(\om) \leq h\} \subset \{ \om \, | \, \|u\|_{ L^q_\al(0, h \wedge \tau;L^p(\Rnp))} \geq R\}.
\]
Using the Chebyshev inequality we obtain that for all $h \in (0,\infty)$,
\[
\bP (\tau = 0)
\le \bP (\{\tau \leq h\})
\leq \frac1{R}  \E \|u\|_{L^q_\al(0,h \wedge \tau;L^p(\Rnp))}
\]
Since $\|u\|_{L ^q_\al(0,h \wedge \tau;L^p(\Rnp))} \le \|u\|_{L ^q_\al(0,h;L^p(\Rnp))} \ri 0$ as $h \to 0$ almost surely, we have
\[
\bP (\tau = 0) = 0 \qand \bP (\tau > 0) = 1.
\]
By the definition of the stopping time $\tau$ we have
\[
\bP(\tau < \infty)
= \E 1_\set{\tau < \infty}
\le \E \Bigg(1_\set{\tau < \infty} \frac{\norm{u}_{L^q_\al (0, T;  L^p (\Rnp))}}{R}\Bigg)
\le \frac{1}{R} \E \Big(\lim_{T \to \infty} \norm{u}_{L^q_\al (0, T \wedge \tau;  L^p (\Rnp))}\Big)
\]
By Fatou's Lemma and \eqref{E41} it is bounded by
\[
\liminf_{T \to \infty} \frac{1}{R} \E \Big(\norm{u}_{L^q_\al (0, T \wedge \tau;  L^p (\Rnp))}\Big)
\le \frac{1}{R} \E \Big(\norm{u}_{L^q_\al (0, \infty;  L^p (\Rnp))}\Big)
\le \frac{5C_1 R^2}{R} = 5C_1 R.
\]
If $R < \frac{1}{5C_1+C_2} \wedge \frac{\ep}{5C_1}$, then
\[
\bP (\tau = \infty) = 1 -\bP (\tau < \infty)
\geq 1 - \ep.
\]
This completes the proof of Theorem \ref{T1}.
\qed
\end{enumerate}

\section{Proof of Theorem \ref{T2}}
\label{S5}
\setcounter{equation}{0}

We divide the proof into a few steps.

\begin{enumerate}[\bf{Step} 1)]
\item
We decompose the solution $u$ of \eqref{E11} formally as the sum
\begin{equation}
\label{E51}
u=v+V + u_2,
\end{equation}
where
\begin{align*}
v(x,t) &= \sum_{i=1}^n \inp{u_{0i}, K(x-\cdot,t)}, \\
V(x,t) &= -\int^t_{0}\int_{\Rnp}  K(x,y, t-s) \bP(\divg (\chi_R^u ( u \times u )(y,s) )\Big) dyds,\\
u_2(x,t) &= \int_0^t \inp{g(\cdot,s),K(x-\cdot, t-s)} dB_s.
\end{align*}
During the proof we shall verify that they are well-defined.
Here $K$ is defined in \eqref{E37}.
We shall give estimates of $v$, $V$, and $u_2$, separately.
\item
For $v$ we have the following estimate.
\begin{lemm}
\label{L51}
For $1 < p < \infty$, $1 \le q \le \infty$, and $\al > 0$,
\[
\norm{v}_{L^q (0, \infty; \dot B^{-2\al}_{pq}(\Rnp))}
\lesssim \norm{u_0}_{\dot B^{-2\al -\frac2q}_{pq0}(\Rnp)}.
\]
\end{lemm}

\begin{proof}
The zero extension of $u_0 \in \dot B_{pq0}^{-2\al-2/q}(\Rnp)$ will be denoted as $\widetilde{u}_0 \in \dot B_{pq}^{-2\al-2/q}(\R)$, which is defined for $f \in \dot B_{p'q'}^{2\al+2/q}(\R)$ by
\[
\inp{\widetilde{u}_0,f} = \inp{u_0,f|_{\Rnp}}.
\]
Note that $\norm{u_0}_{\dot B_{pq0}^{-2\al-2/q}(\Rnp)}$ is comparable to $\norm{\widetilde{u}_0}_{\dot B_{pq}^{-2\al-2/q}(\R)}$.
Following the proof of Theorem 4.1 in \cite{MR3474350}, we can decompose $v$ as the sum
\[
v(x,t) = v_{1}(x,t) + v_{2}(x,t) + v_{3}(x,t)
\]
where $v_1$, $v_2$ and $v_3$ are defined by
\begin{align*}
v_{1i}(x,t)
&= \Ga_t * \widetilde{u}_{0i}(x) - \Ga_t^* * \widetilde{u}_{0i}(x), \\
v_{2i}(x,t) 
&= 4 \frac{\p }{\p  x_i} \int_{\Rnp} \sum_{j=1}^{n-1} \frac{\p  }{\p y_j} (N(x-y) - N(x-y^*) ) \Ga_t^* * u_{0j} (y) dy \\ 
&\quad -2 \frac{\p }{\p  x_i} \int_{\Rnp} \frac{\p  }{\p  y_n} (N(x-y) - N(x-y^*) ) \sum_{j=1}^{n-1} R_j' \Ga_t^* * u_{0j}  (y) dy, \\
v_{3i}(x,t) &= -4\delta_{in}   \sum_{j=1}^{n-1} R'_j \Ga^*_t* \widetilde{u}_{0j}(x).
\end{align*}
Here $R_j'$ denotes the Riesz transform in $\bR^{n-1}$.
Since $R'_j$ is an $L^p(\bR^{n-1})$-multiplier, $R'_j$ is bounded in $L^p(\R)$ by Fubini's theorem.
Thus, we have
\begin{equation}
\label{E52}
\norm{v_1(t)}_{\dot B^{-2\al}_{pq} (\Rnp)} + \norm{v_3(t)}_{\dot B^{-2\al}_{pq} (\Rnp)}
\lesssim \|\Ga_t * \widetilde{u}_{0}\|_{\dot B^{-2\al}_{pq} (\R)} + \|\Ga_t^* * \widetilde{u}_{0}\|_{\dot B^{-2\al}_{pq} (\R)}
\end{equation}
and from Lemma \ref{L22}
\begin{align}
\label{E53}
\|v_2(t)\|_{\dot B^{-2\al}_{pq} (\Rnp)} \lesssim \|\Ga_t^* * \widetilde{u}_{0}\|_{\dot B^{-2\al}_{pq} (\R)}.
\end{align}
Using the estimates \eqref{E52}, \eqref{E53}, and Lemma \ref{L24}, we obtain that
\begin{align*}
\norm{v}_{L^q (0, \infty; \dot B^{-2\al}_{pq}(\Rnp))}
&\le \norm{v_1}_{L^q (0, \infty; \dot B^{-2\al}_{pq}(\Rnp))}
+ \norm{v_2}_{L^q (0, \infty; \dot B^{-2\al}_{pq}(\Rnp))}
+ \norm{v_3}_{L^q (0, \infty; \dot B^{-2\al}_{pq}(\Rnp))} \\
&\lesssim \norm{\Ga_t * \tilde u_0}_{L^q(0, \infty;  \dot B^{-2\al}_{pq} (\R)) }
+ \norm{\Ga_t^* * \tilde u_0}_{L^q(0, \infty;  \dot B^{-2\al}_{pq} (\R)) } \\
&\lesssim \norm{\tilde u_0}_{\dot B^{-2\al -\frac2q}_{pq}(\R)}
\lesssim \norm{u_0}_{\dot B^{-2\al -\frac2q}_{pq}(\Rnp)}.
\end{align*}
This completes the proof of Lemma \ref{L51}.
\end{proof}

\item
For $V$ we have the following estimate.
\begin{lemm}
\label{L52}
For $1 < p < \infty$, $1 \le q \le \infty$, and $\al > 0$,
\[
\|V\|_{L^q(0, \infty;\dot B^{-2\al}_{pq} (\Rnp))}
\lesssim \| u\|^2_{L^{q}_{ \al} (0, \tau(\om); L^{p} (\Rnp))}.
\]
\end{lemm}

\begin{proof}
We can decompose $V$ as the sum
\[
V(x,t) = V_{1}(x,t) + V_{2}(x,t) + V_{3}(x,t),
\]
where $V_1$, $V_2$ and $ V_3$ are defined by
\begin{align*}
V_{1i}(x,t)
&= D_{x'} \big({\mathcal U} F'(x,t) -  {\mathcal U}^*F'(x,t) \big) \\
&:= \int_0^t \int_{\Rnp} D_{y'} (\Ga(x-y, t-s)-\Ga(x-y^*, t-s) )F'_i(y, s) dy ds, \\
V_{2i}(x,t) 
&= 4 \frac{\p }{\p  x_i} \int_{\Rnp} \sum_{j=1}^{n-1} \frac{\p  }{\p y_j} (N(x-y) - N(x-y^*)) D_{y'} {\mathcal U}^* F'_j (y,t) dy \\
&\quad - 2 \frac{\p }{\p  x_i} \int_{\Rnp} \frac{\p  }{\p  y_n} (N(x-y) - N(x-y^*) ) \sum_{j=1}^{n-1} R_j' D_{y'} {\mathcal U}^* F'_j (y,t) dy, \\
V_{3i}(x,t) &= -4\delta_{in}   \sum_{j=1}^{n-1}  R'_j D_{x'} {\mathcal U}^* F'_j(x,t),
\end{align*}
with $F_{ij} = \chi_R^u u_i u_j$.
Notice that
\begin{equation}
\label{E54}
\begin{split}
\norm{V_k(t)}_{\dot B^{-2\al}_{pq} (\Rnp)}
&\le \norm{V_1(t)}_{\dot B^{-2\al}_{pq} (\Rnp)}
+ \norm{V_2(t)}_{\dot B^{-2\al}_{pq} (\Rnp)}
+ \norm{V_3(t)}_{\dot B^{-2\al}_{pq} (\Rnp)} \\
&\lesssim \|D_{x'}{\mathcal U}F'(t)\|_{\dot B^{-2\al }_{pq} (\Rnp)} + \|D_{x'}{\mathcal U}^*F'(t)\|_{\dot B^{-2\al }_{pq} (\Rnp)}.
\end{split}
\end{equation}
If $p, q, \al, p_1, q_1, \al_1$ satisfy $0 < \frac12 - \frac{1}{2}(\frac{n}{p_1'}-\frac{n}{p'}) - \al < 1$, $1 < q_1 \le q < \infty$, $0 \le \al_1 < 1-\frac1{q_1}$, and $\bigg(\frac{n}{p_1}-\frac{n}{p}\bigg) + \bigg(\frac{2}{q_1}-\frac{2}{q}\bigg) + 2(\al_1-\al) = 1$, then
\begin{equation}
\label{E55}
\| D_{x'}{\mathcal U}F' \|_{ L^q (0, \infty; \dot B^{-2\alpha}_{pq}(\Rnp) )}
+ \| D_{x'}{\mathcal U}^* F' \|_{ L^q (0, \infty; \dot B^{-2\alpha}_{pq}(\Rnp) )}
\lesssim \| F'\|_{L^{q_1}_{ \alpha_1} (0, \infty; L^{p_1} (\Rnp))}.
\end{equation}
Since the proof of this technical estimate \eqref{E55} is a little bit long, we postpone its proof to the next step.
Combining \eqref{E54}, \eqref{E55}, and \eqref{E22}, we obtain that
\[
\norm{V}_{L^q(0, \infty;\dot B^{-2\al}_{pq} (\Rnp))}
\lesssim \| F\|_{L^{q_1}_{ \alpha_1} (0, \infty; L^{p_1} (\Rnp))}.
\]
Finally, we can take $q_1=\frac{q}{2}$, $p_1=\frac{p}{2}$, and $\al_1=2\al$ so that the exponents $p, q, \al, p_1, q_1, \al_1$ satisfy every conditions on the exponents and
\[
\| F\|_{L^{q_1}_{ \alpha_1} (0, \infty; L^{p_1} (\Rnp))}
\lesssim \| u\|^2_{L^{q}_{ \al} (0, \tau(\om); L^{p} (\Rnp))}.
\]
This completes the proof of Lemma \ref{L52}.
\end{proof}

\item
Now we prove the technical estimate \eqref{E55}.
Since the estimate of $D_{x'}{\mathcal U}^* F'$ can be derived by exactly the same way as that of $D_{x'}{\mathcal U}F'$, we focus on estimating the term $D_{x'}{\mathcal U}F'$, which is done by duality argument.
If $\psi \in L^{q'}(0, \infty;\dot B^{2\alpha}_{p'q'0} (\Rnp))$, then by Fubini's theorem
\begin{align*}
&\intzi \int_{\Rnp} D_{x'}{\mathcal U} F'(x,t) \psi(x,t) dxdt \\
&= \intzi \int_{\R} D_{x'}{\mathcal U} F'(x,t) \widetilde{\psi}(x,t) dxdt \\
&= \intzi \int_{\R} \Bigg(\int_0^t \int_{\Rnp} F'(y,s) D_{y'} \Ga(x-y, t-s) dy ds\Bigg) \widetilde{\psi}(x,t) dx dt \\
&= \intzi \int_0^t \int_{\Rnp} F'(y,s) \int_{\R} D_{y'} \Ga(x-y, t-s) \widetilde{\psi}(x,t) dx dy ds dt.
\end{align*}
Integrating by parts and using H\"older's inequality, we have
\begin{align*}
&\Bigg| \int_{\Rnp} F'(y,s) \int_{\R} D_{y'} \Ga(x-y, t-s) \widetilde{\psi}(x,t) dx dy \Bigg| \\
&= \Bigg| \int_{\Rnp} F'(y,s) \int_{\R} \Ga(y-x, t-s) D_{x'} \widetilde{\psi}(x,t) dx dy \Bigg| \\
&= \Bigg| \int_{\Rnp} F'(y,s) \Ga_{t-s}*D' \widetilde{\psi}(y,t) dy \Bigg| \\
&\le \norm{F'(s)}_{L^{p_1}(\Rnp)} \norm{\Ga_{t-s}*D' \widetilde{\psi}(t)}_{L^{p_1'}(\Rnp)}.
\end{align*}
Using the semigroup property of the heat kernel and Young's convolution inequality, we get
\begin{align*}
\norm{\Ga_{t-s}*D' \widetilde{\psi}(t)}_{L^{p_1'}(\Rnp)}
&\le \norm{\Ga_{t-s}*D' \widetilde{\psi}(t)}_{p_1'} \\
&= \norm{\Ga_{(t-s)/2}*\Ga_{(t-s)/2}*D' \widetilde{\psi}(t)}_{p_1'} \\
&\le \norm{\Ga_{(t-s)/2}}_{L^{r}(\R)} \norm{\Ga_{(t-s)/2}*D' \widetilde{\psi}(t)}_{p'} \\
&\lesssim (t-s)^{-\frac{n}{2}(1-\frac{1}{r})} \norm{\Ga_{(t-s)/2}*D' \widetilde{\psi}(t)}_{p'} \\
&= (t-s)^{\frac{n}{2}(\frac{1}{p_1'}-\frac{1}{p'})} \norm{\Ga_{(t-s)/2}*D' \widetilde{\psi}(t)}_{p'},
\end{align*}
where $1+1/p_1'=1/r+1/p'$.
Since Lemma \ref{L25} yields
\[
\norm{\Ga_{(t-s)/2}*D' \widetilde{\psi}(t)}_{p'}
\lesssim (t-s)^{-\frac12 + \al} \|D' \widetilde{\psi}(t)\|_{\dot B^{-1 +2\al}_{p'q'}(\R)}
\lesssim (t-s)^{-\frac12 + \al} \|\widetilde{\psi}(t)\|_{\dot B^{2\al}_{p'q'}(\R)},
\]
we obtain that
\[
\norm{\Ga_{t-s}*D' \psi(t)}_{L^{p_1'}(\Rnp)}
\lesssim (t-s)^{\frac{n}{2}(\frac{1}{p_1'}-\frac{1}{p'})-\frac12 + \al} \|\psi(t)\|_{\dot B^{2\al}_{p'q'}(\Rnp)}.
\]
Combining the estimate above and then using H\"older's equality, we get
\begin{align*}
&\Bigg| \intzi \int_{\Rnp} D_{x'}{\mathcal U}F'(x,t) \psi(x,t) dxdt \Bigg| \\
&\le \intzi \int_0^t \norm{F'(s)}_{L^{p_1}(\Rnp)} (t-s)^{\frac{n}{2}(\frac{1}{p_1'}-\frac{1}{p'})-\frac12 + \al} \|\psi(t)\|_{\dot B^{2\al}_{p'q'}(\Rnp)} ds dt \\
&= \intzi I_\lambda f(t) \|\psi(t)\|_{\dot B^{2\al}_{p'q'}(\Rnp)} dt \\
&\le \norm{I_\lambda f}_{L^q(0, \infty)} \norm{\psi}_{L^{q'}(0, \infty; {\dot B^{2\al}_{p'q'}(\Rnp)})},
\end{align*}
where
\[
\lambda := -\frac{n}{2}(\frac{1}{p_1'}-\frac{1}{p'}) + \frac12 - \al \qand f(s) := \norm{F'(s)}_{L^{p_1}(\Rnp)}.
\]
If $0 < \lambda < 1$, $1 < q_1 \le q < \infty$, $0 \le \al_1 < 1/q_1'$, and $1+1/q=1/q_1+\lambda+\al_1$, then by Lemma \ref{L26}
\[
\norm{I_\lambda f}_{L^q(0, \infty)}
\lesssim \norm{f}_{L^{q_1}_{ \alpha_1} (0, \infty)}
= \norm{F'}_{L^{q_1}_{ \alpha_1} (0, \infty; L^{p_1} (\Rnp))}.
\]
Thus, we have
\[
\Bigg| \intzi \int_{\Rnp} D_{x'}{\mathcal U}F'(x,t) \psi(x,t) dxdt \Bigg|
\lesssim \norm{F'}_{L^{q_1}_{ \alpha_1} (0, \infty; L^{p_1} (\Rnp))}
\norm{\psi}_{L^{q'}(0, \infty; {\dot B^{2\al}_{p'q'}(\Rnp)})}
\]
and therefore by duality $\| D_{x'}{\mathcal U}F' \|_{ L^q (0, \infty; \dot B^{-2\alpha}_{pq}(\Rnp) )}
\lesssim \| F'\|_{L^{q_1}_{ \alpha_1} (0, \infty; L^{p_1} (\Rnp))}$.
This completes the proof of the estimate \eqref{E55}.

\item
For $u_2$ we have the following estimate.
\begin{lemm}
\label{L53}
For $1 < p < \infty$, $1 \le q \le \infty$, and $\al > 0$,
\[
\E \intzi  \|u_2(t)\|^q_{\dot B^{-2\al}_{pq}(\Rnp)}dt
\lesssim  \E \intzi  \| g(t)\|^q_{\dot B^{-2\al-1}_{pq0}(\Rnp)}dt.
\]
\end{lemm}

\begin{proof}
We can decompose $u_2$ as the sum
\[
u_{2i}(x,t) = u_{21i}(x,t) + u_{22i}(x,t) + u_{23i}(x,t),
\]
where $u_{21i}$, $u_{22i}$ and $ u_{23i}$ are defined by
\begin{align*}
u_{21i}(x,t)
&={\mathcal U} \widetilde{g}(x,t) - {\mathcal U}^* \widetilde{g}(x,t)\\
&:= \int_0^t \Ga_{t-s} * \widetilde{g_i}(\cdot,s)(x) - \Ga_{t-s}^* * \widetilde{g_i}(\cdot,s)(x) dB_s, \\
u_{22i}(x,t) 
&= 4 \frac{\p }{\p  x_i}     \int_{\Rnp} \sum_{j=1}^{n-1} \frac{\p  }{\p  y_j} (N(x-y) - N(x-y^*) ) {\mathcal U}^* \widetilde{g}_j (y,s) dy \\
&\quad -2 \frac{\p }{\p x_i} \int_{\Rnp} \frac{\p  }{\p  y_n} (N(x-y) - N(x-y^*) ) \sum_{j=1}^{n-1} R_j' {\mathcal U}^* \widetilde{g}_j (y,s) dy \\
u_{23i}(x,t) &= -4\delta_{in}   \sum_{j=1}^{n-1} R'_j {\mathcal U}^* \widetilde{g}_j(x,t).
\end{align*}
Thus, by Lemma \ref{L22} we get
\[
\sum_{k=1}^3\|u_k(t)\|_{\dot B^{-2\al}_{pq} (\Rnp)}
\lesssim \|{\mathcal U} \widetilde{g}(t)\|_{\dot B^{-2\al}_{pq} (\R)} + \|{\mathcal U}^* \widetilde{g}(t)\|_{\dot B^{-2\al}_{pq} (\R)}.
\]
If we prove that for $j \in \N$,
\begin{equation}
\label{E56}
\E \intzi  2^{qj}\|{\De_j \mathcal U}\widetilde{g}(t)\|^q_p dt
\lesssim \E\int^\infty_0\|\De_j \widetilde{g}(t)\|^q_p dt,
\end{equation}
then we obtain that
\begin{align*}
\E \intzi  \| {\mathcal U} \widetilde{g}(t)\|^q_{\dot B^{-2\al}_{pq}(\R)}dt
&= \E \intzi  \sum_{-\infty< j< \infty} 2^{-2\al q j} \|\De_j {\mathcal U}\widetilde{g}(t)\|^q_p dt\\
&\lesssim \E\int^\infty_0 \sum_{-\infty< j< \infty} 2^{-( 1 +2\al )q j}\|\De_j \widetilde{g}(t)\|^q_pdt\\
&= \E \intzi  \| \widetilde{g}(t)\|^q_{\dot B^{-2\al-1}_{pq0}(\R)}dt.
\end{align*}
\end{proof}

\item
Fianlly, we prove the estimate \eqref{E56}.
Notice that $p > n \ge 2$.
The Burkholder-Davis-Gundy inequality (Section 2.7 in \cite{MR1661766}) yields
\begin{align*}
&\E \intzi \|{\De_j \mathcal U}\widetilde{g}(t)\|^q_p dt \\
&= \E \intzi  \Big(\intRn \Bigg|\int_0^t \De_j (\Ga_{t-s} * \widetilde{g})(x,s) dB_s \Bigg|^p dx \Big)^{\frac{q}{p}} dt \\
&\lesssim \E \intzi \left(\intRn \left(\int_0^t |\De_j (\Ga_{t-s} * \widetilde{g})(x,s)|^2 ds\right)^{p/2} dx\right)^{q/p} dt.
\end{align*}
By Minkowski's integral inequality and Lemma \ref{L23} we obtain
\begin{align*}
&\E \intzi \left(\intRn \left(\int_0^t |\De_j (\Ga_{t-s} * \widetilde{g})(x,s)|^2 ds\right)^{p/2} dx\right)^{q/p} dt \\
&\le \E \intzi \left(\int_0^{t} \left(\intRn |\De_j (\Ga_{t-s} * \widetilde{g})(x,s)|^p dx\right)^{2/p} ds\right)^{q/2} dt \\
&= \E \intzi \left( \int_0^{t} \norm{\De_j (\Ga_{t-s} * \widetilde{g})(s)}_p^2 ds\right)^{q/2} dt \\
&\lesssim \E \intzi \left( \int_0^{t} \exp(-c(t-s)2^{2j}) \norm{\De_j \widetilde{g}(s)}_p^2 ds\right)^{q/2} dt
\end{align*}
for some positive constant $c$.
By Young's convolution inequality we have
\begin{align*}
&\E \intzi \left(\int_0^{t} \exp(-c(t-s)2^{2j}) \norm{\De_j \widetilde{g}(s)}_p^2 ds\right)^{q/2} dt \\
&\le \E \left(\intzi \exp(-ct2^{2j}) dt\right)^{q/2} \intzi \norm{\De_j \widetilde{g}(t)}_p^q dt \\
&= (c2^{2j})^{-q/2} \E \intzi \norm{\De_j \widetilde{g}(t)}_p^q dt.
\end{align*}
Thus, we get \eqref{E56}.

From \eqref{E51}, Lemma \ref{L51}, Lemma \ref{L52}, and Lemma \ref{L53}, we get the desired result.
This completes the proof of Theorem \ref{T2}.
\qed
\end{enumerate}

\section*{Acknowledgement}

T. C. has been supported by the National Research Foundation of Korea No. 2020R1A2C1A01102531.
M. Y. has been supported by the National Research Foundation of Korea No. 2016R1C1B2015731 and No.~2015R1A5A1009350.

\end{document}